\documentclass[10pt]{amsart}

\usepackage{amsmath,amssymb,amsthm}
\usepackage{mathrsfs}
\usepackage{enumitem}
\usepackage{booktabs}
\usepackage[T2A,T1]{fontenc}
\usepackage[utf8]{inputenc}
\usepackage{graphicx}
\usepackage[margin=1.15in]{geometry}
\usepackage{microtype}
\usepackage[dvipsnames]{xcolor}
\usepackage[colorlinks=true,linkcolor=RoyalBlue,citecolor=ForestGreen,urlcolor=BrickRed]{hyperref}
\usepackage{cleveref}
\theoremstyle{plain}
\newtheorem{theorem}{Theorem}[section]
\newtheorem{lemma}[theorem]{Lemma}

\newtheorem{corollary}[theorem]{Corollary}

\theoremstyle{definition}

\theoremstyle{remark}

\newcommand{\Sha}{\text{\fontencoding{T2A}\selectfont\CYRSH}}
\newcommand{\Ad}{\operatorname{Ad}}
\newcommand{\Cov}{\operatorname{Cov}}
\newcommand{\Corr}{\operatorname{Corr}}

\newcommand{\FF}{\mathbb{F}}
\newcommand{\QQ}{\mathbb{Q}}
\newcommand{\ZZ}{\mathbb{Z}}
\newcommand{\RR}{\mathbb{R}}

\newcommand{\E}{\mathbb{E}}
\newcommand{\Fam}{\mathcal{F}}
\newcommand{\WE}{W(E_8)}
\newcommand{\detp}{{\det}'}
\newcommand{\disc}{\operatorname{disc}}
\newcommand{\mult}{\operatorname{mult}}

\newcommand{\rank}{\operatorname{rank}}
\newcommand{\ord}{\operatorname{ord}}
\newcommand{\covol}{\operatorname{covol}}

\usepackage[font=small,labelfont=bf]{caption}
\usepackage{needspace}
\makeatletter
\renewcommand{\section}{\@startsection{section}{1}%
  \z@{.7\linespacing\@plus\linespacing}{.5\linespacing}%
  {\normalfont\bfseries\centering}}
\renewcommand{\subsection}{\@startsection{subsection}{2}%
  \z@{.5\linespacing\@plus.7\linespacing}{-.5em}%
  {\normalfont\bfseries}}
\makeatother

\begin{document}

\title[Murmurations, periods, and local factors]{Murmurations, periods, and local factors}

\author{Dane Wachs}
\address{The University of Arizona, Tucson, AZ 85721}
\date{October 2026}
\email{wachs@arizona.edu}

\begin{abstract}
This note replaces arXiv:2603.22807v1 and withdraws its three theorems, which do not stand as stated. Theorem~A, that the order of the Tate--Shafarevich group of $y^2=x^3+x+D(t)$ over $\mathbb{F}_q(t)$ is determined by the cyclotomic type of the $L$-polynomial, holds for $\deg D=5$ at rank $0$ but fails at positive rank; we give a counterexample and cite the correct statement. Theorem~B is Tur\'an's inequality for modified Bessel functions; we prove a version for general even weights and a class of increasing tilts that includes the Euler factor, which gives the Euler-factor tilt at every prime. Theorem~C, a limit law for the covariance of $a_p$ with the real period at fixed $L(f,1)$, was not proved, and we explain why. We also recompute the empirical section of the first version from Cremona's tables, correct its numbers, withdraw those we could not reconstruct, and show that its sign flip at $p=3$ and its ``100\% Tamagawa channel'' are artefacts of decile conditioning and of the rank-$0$ BSD formula.
\end{abstract}

\maketitle

\section{Introduction}\label{sec:intro}

This version replaces the first version of this paper, arXiv:2603.22807v1 (March 2026; ``v1'' below). v1 asked how the invariants of the Birch--Swinnerton-Dyer formula enter murmuration statistics \cite{HLOP}, over $\FF_q(t)$ and over $\QQ$. None of its three theorems, Theorems~A, B and C (numbered $1$--$3$ in its body, whose numbering we use for its other items), stands as v1 presents it, and its empirical section contains wrong numbers. This version withdraws the theorems, states what is true instead, and corrects the numbers.

\emph{Withdrawn.} Theorem~A (over $\FF_q(t)$, ``$|\Sha|$ is an invariant of the cyclotomic type of the $L$-polynomial'') is false at positive rank; its proof used the Birch--Swinnerton-Dyer formula for rank-$0$ curves only (Section~\ref{sec:ff}). Theorem~B (positivity of a tilted Sato--Tate covariance, derived from what v1 called a new inequality for modified Bessel functions) is true but classical: by the identity \eqref{eq:turan-identity} it is Tur\'an's inequality, and v1's Bessel inequality is Amos's (Section~\ref{sec:tilt}). Theorem~C (a limit law $C(c)/\sqrt p$ for $\Cov(a_p(f),\Omega_{E_f})$ given $L(f,1)\approx c$) was not proved: $\Omega_{E_f}$ is undefined for most newforms, the period relation used is wrong in direction, and the conditioning on $L(f,1)$ is not controlled (Section~\ref{sec:thmC}).

\emph{What replaces them.} For Theorem~A, Theorems~A and A$'$ of \cite{WachsFF} (Theorem~\ref{thm:ff}). For Theorem~B, an elementary lemma (Lemma~\ref{lem:Bprime}) that proves at every prime the Euler-factor tilt that v1 stated as a theorem but checked only numerically for $p\le97$: for $\lambda>0$, $\alpha_p(\lambda)>0$ and $\Cov(x,x^2)>0$, and for both signs of $\lambda$, $\Cov(\cos\theta,\cos2\theta)$ has the sign of $\lambda$ (Corollary~\ref{cor:euler-tilt}). Nothing replaces Theorem~C. Section~\ref{sec:empirical} recomputes v1's empirical section; \S\ref{sec:rest} lists the other material of v1 that is withdrawn and the empirical statements of v1 that survive.

\emph{The $|\Sha|$ thesis.} v1 concluded that ``the $|\Sha|$-modulation of murmurations [of the first version of \cite{WachsQ}] is a consequence of the BSD identity linking $|\Sha|$ to local factors''. This holds only in a tautological sense. At rank $0$ the analytic order of $\Sha$ in Cremona's tables is defined by the Birch--Swinnerton-Dyer formula, $|\Sha|=L(E,1)\,|T|^2/(\Omega\prod c_p)$ with $|T|=|E(\QQ)_{\rm tors}|$; so $\log|\Sha|$ is an exact linear combination of $\log L(E,1)$, $\log|T|$, $\log\Omega$ and $\log\prod c_p$, and no statistic of the traces can carry information about $|\Sha|$ that these four quantities do not carry \cite[\S1.2]{WachsQ}. v1's ``the same local-factor mechanism operates discretely over function fields'' has no content in the family of Section~\ref{sec:ff}, where every Tamagawa number is $1$ \cite[Thm.~A(1)]{WachsFF}.

\section{Function fields: Theorem A}\label{sec:ff}

Let $q=p^f$ with $p\ge5$, $K=\FF_q(t)$, $\Fam_5(q)=\{D\in\FF_q[t]$ monic, $\deg D=5$, $4+27D^2$ squarefree$\}$ and $E_D\colon y^2=x^3+x+D(t)$. Frobenius acts on the Mordell--Weil lattice $E_D(\bar\FF_q(t))\cong E_8$ by an element $w_D\in\WE$, well defined up to conjugacy \cite[Thm.~A(2)]{WachsFF}. v1's Theorem~A, stated for $D$ monic squarefree of any degree, has three clauses: (a) every $L$-polynomial of the family factors into cyclotomic polynomials; (b) if $4+27D^2$ is squarefree, $|\Sha(E_D)|$ is determined by the cyclotomic type; (c) in particular, $|\Sha|$-conditioned murmuration densities are type-weighted averages ``with no within-type zero displacement''. For $\Fam_5(q)$ the correct statement is the following.

\begin{theorem}[{\cite[Thm.~A, Thm.~A$'$]{WachsFF}}; proven there]\label{thm:ff}
Let $D\in\Fam_5(q)$, $w=w_D$, $r=\mult_1(w)$ the multiplicity of the eigenvalue $1$, and $E_8^w=\ker(1-w)$.
\begin{enumerate}[label=\textup{(\roman*)},leftmargin=2em]
\item $L(E_D,T)=\det(1-q\,w\,T\mid E_8)$, so the unitarised polynomial $z^8L(E_D,1/(qz))$ is the characteristic polynomial $\chi_w$ of $w$, a product of cyclotomic polynomials, and $\rank E_D(K)=r=\ord_{T=1/q}L(E_D,T)$.
\item $\Sha(E_D)\cong(E_8/(1-w)E_8)_{\mathrm{tors}}$ and $|\Sha(E_D)|=\detp(1-w)/\disc(E_8^w)$, where $\detp(1-w)=\prod_{\lambda\ne1}(1-\lambda)$ over the eigenvalues $\lambda\ne1$ of $w$.
\item At rank $0$, $|\Sha(E_D)|=\chi_w(1)=L(E_D,1/q)$, and the $30$ rank-$0$ classes of $\WE$ have $30$ distinct characteristic polynomials.
\item At rank $\ge1$ this fails: $\Phi_1\Phi_2\Phi_4\Phi_8$, $\Phi_1^2\Phi_2^2\Phi_3\Phi_6$, $\Phi_1^2\Phi_2^2\Phi_4^2$, $\Phi_1^3\Phi_2^3\Phi_4$, $\Phi_1^4\Phi_2^4$ are each the characteristic polynomial of two or three classes whose discriminants $\disc(E_8^w)$ are not all equal.
\end{enumerate}
\end{theorem}

\begin{corollary}[{\cite[Cor.~1.1]{WachsFF}}; proven there, the counts computed]\label{cor:z8}
$z^8-1=\Phi_1\Phi_2\Phi_4\Phi_8$ \textup{(}rank $1$, $\detp(1-w)=8$\textup{)} is the characteristic polynomial of a class $C$ with $E_8^w=\ZZ v$, $\langle v,v\rangle=8$, $\Sha=0$, and of a class $C'$ with $E_8^w=\ZZ v'$, $v'$ a root \textup{(}$\disc E_8^w=2$\textup{)}, $\Sha\cong(\ZZ/2)^2$. The curves with $w_D\in C'$ are exactly those with an $\FF_q$-rational section $(x,y)$, $\deg x\le2$, $\deg y\le3$. At $q=7$ the $1{,}274$ curves of type $z^8-1$ split as $1{,}120$ in $C$ and $154$ in $C'$.
\end{corollary}

So curves with the same $L$-function can have $|\Sha|=1$ and $|\Sha|=4$ (both classes contain curves with $D$ squarefree, as v1's family requires: $952$ and $112$ of them; ancillary file \texttt{ff/z8\_squarefree\_check.py}), and clause (b) is false at rank $\ge1$. v1's proof of (b) used the Birch--Swinnerton-Dyer formula for rank-$0$ curves only; at rank $0$, (b) is Theorem~\ref{thm:ff}(iii). The words of clause (c) follow from (a) alone, since a type fixes the $L$-polynomial; for $D\in\Fam_5(q)$ they hold at every rank. What fails at positive rank is v1's formula for the $|\Sha|$-conditioned density (v1, \S2.3), $\delta_s(d)=\sum_{\lambda:|\Sha|_\lambda=s}N_\lambda(q)p_d(\lambda)/\sum_{\lambda:|\Sha|_\lambda=s}N_\lambda(q)$, which assigns one value $|\Sha|_\lambda$ to each type $\lambda$; the correct weights are the numbers of curves of type $\lambda$ with $|\Sha|=s$.

Clause (a) holds for $D\in\Fam_5(q)$ by Theorem~\ref{thm:ff}(i); we do not treat other $D$ here. It fails in degree $7$. v1 deduced it from the Riemann hypothesis and Kronecker's theorem, which needs the unitarised polynomial to have integer coefficients. If $L(E_D,T)=1+c_1T+\cdots$, then $c_1$ is the sum of the traces $a_v$ over the places $v$ of degree $1$, and the coefficient of $z^dL(E_D,1/(qz))$ ($d=\deg L$) next to the leading one is $c_1/q$. Take $q=7$ and $D=t^7+t^6+1$; $D$ and $4+27D^2$ are squarefree. At $t=\infty$, with $t=1/u$, $x=X/u^4$, $y=Y/u^6$, the model $Y^2=X^3+u^8X+u^5D^*(u)$, $D^*(u)=u^7D(1/u)$, is minimal ($\ord_u a_6=5$) and has additive reduction ($\ord_u\Delta=10$, $\ord_u c_4=8$), so $a_\infty=0$. At the finite places $t=0,\dots,6$ the traces are $3,2,-2,1,-3,0,3$. So $c_1=4\not\equiv0\pmod7$: the unitarised polynomial is not integral, and so not a product of cyclotomic polynomials. In the whole family at $q=7$ ($D$ monic of degree $7$, $D$ and $4+27D^2$ squarefree: $513{,}870$ polynomials), $440{,}426$ ($85.7\%$) have $c_1\not\equiv0\pmod7$. Two computations, independent at the finite places (one by character sums with its own squarefree test, one by PARI's point counts with a split/non-split test at the nodes and PARI's squarefree test) and sharing the value $a_\infty=0$ derived above, agree on the example, on these counts and on the distribution of $c_1$ (ancillary files \texttt{ff/deg7\_c1.py} and \texttt{ff/deg7\_c1.gp}).

\section{Tilted Sato--Tate covariances: Theorem B}\label{sec:tilt}

For $\lambda\in\RR$ let $\mu_\lambda(dx)=Z(\lambda)^{-1}\sqrt{1-x^2}\,e^{\lambda x}dx$ on $[-1,1]$, with $Z(\lambda)=\int_{-1}^1\sqrt{1-x^2}\,e^{\lambda x}dx=\pi I_1(\lambda)/\lambda$, where $I_\nu$ is the modified Bessel function of the first kind. In Sato--Tate coordinates $x=\cos\theta$, $\mu_\lambda\propto\sin^2\theta\,e^{\lambda\cos\theta}d\theta$ on $[0,\pi]$ and $\cos2\theta=2x^2-1$. v1's Theorem~B states that $\Cov_{\mu_\lambda}(\cos\theta,\cos2\theta)=2\Cov_{\mu_\lambda}(x,x^2)$ is positive for $\lambda>0$ and negative for $\lambda<0$.

\begin{lemma}[Tur\'an, Amos; proven]\label{lem:turan}
For every $\lambda>0$, $\E_{\mu_\lambda}[x^2]=1-3I_2(\lambda)/(\lambda I_1(\lambda))$ and
\begin{equation}\label{eq:turan-identity}
\Cov_{\mu_\lambda}(x,x^2)=\frac{d}{d\lambda}\E_{\mu_\lambda}[x^2]=\frac{3\bigl(I_2(\lambda)^2-I_1(\lambda)I_3(\lambda)\bigr)}{\lambda I_1(\lambda)^2}>0,
\end{equation}
and $\Cov_{\mu_\lambda}(x,x^2)<0$ for $\lambda<0$. The positivity is Tur\'an's inequality $I_2^2>I_1I_3$ \cite{TN1951,Amos1974,Baricz2010}. Equivalently, for $z>0$,
\begin{equation}\label{eq:amos}
\Bigl(\frac{I_0(z)}{I_1(z)}\Bigr)^2>1+\frac4{z^2}
\iff
\frac{I_1(z)}{I_0(z)}<\frac{z}{\sqrt{z^2+4}}
\iff
\frac{I_2(z)}{I_1(z)}>\frac{z}{2+\sqrt{z^2+4}},
\end{equation}
which is Amos's inequality \cite[eq.~(11)]{Amos1974} at $\nu=0$ \textup{(}upper bound\textup{)} and $\nu=1$ \textup{(}lower bound\textup{)}.
\end{lemma}

\begin{proof}
The integral representation $I_\nu(z)=\frac{(z/2)^\nu}{\sqrt\pi\,\Gamma(\nu+\frac12)}\int_{-1}^1(1-t^2)^{\nu-1/2}e^{zt}dt$ at $\nu=1,2$ and $x^2\sqrt{1-x^2}=\sqrt{1-x^2}-(1-x^2)^{3/2}$ give $\E_\lambda[x^2]$. As $\{\mu_\lambda\}$ is an exponential family with sufficient statistic $x$, $\frac{d}{d\lambda}\E_\lambda[x^2]=\Cov_\lambda(x,x^2)$. With $R=I_2/I_1$, the relations $I_1'=I_0-I_1/\lambda$, $I_2'=I_1-2I_2/\lambda$, $I_0=I_2+2I_1/\lambda$, $I_3=I_1-4I_2/\lambda$ give $\lambda R'-R=-\lambda(I_2^2-I_1I_3)/I_1^2$, hence $\frac{d}{d\lambda}\E_\lambda[x^2]=-3\frac{d}{d\lambda}[R/\lambda]=3(I_2^2-I_1I_3)/(\lambda I_1^2)$. Positivity for $\lambda>0$ is Tur\'an's inequality at index $2$, and $x\mapsto-x$ maps $\mu_\lambda$ to $\mu_{-\lambda}$. Finally $\frac{d}{d\lambda}[R/\lambda]<0\iff\lambda R^2+4R-\lambda>0\iff R>z/(2+\sqrt{z^2+4})$ at $z=\lambda$, and $R=I_0/I_1-2/z$ turns this into $(I_0/I_1)^2>1+4/z^2$.
\end{proof}

So v1's Theorem~B is Tur\'an's inequality, and the inequality $(I_0/I_1)^2>1+4/z^2$ from which v1 derived it (v1, Appendix~A) is Amos's. Two of v1's constants are wrong. In its proof (v1, Appendix~A, Step~4(a)), $\varphi(z):=z^2((I_0/I_1)^2-1)-4=z^4/48+O(z^6)$, not $z^4/16$; and in its Remark~9, $\Cov_{\mu_\lambda}(x,x^2)=3/\lambda^2+O(\lambda^{-3})$ as $\lambda\to\infty$, not $3/(2\lambda^2)$. $\Cov_{\mu_\lambda}(x,x^2)=\lambda/16+O(\lambda^3)$ (v1, Remark~8) is correct. v1's Remark~7 says that $(I_0/I_1)^2>1+4/z^2$ ``is related to, but does not follow from'' Tur\'an's inequality, and displays $(I_0/I_1)^2>(I_{-1}/I_1)(I_0/I_2)$; as $I_{-1}=I_1$ the right side is $I_0/I_2$, and the displayed inequality is false: at $z=1$, $(I_0/I_1)^2=5.02$ and $I_0/I_2=9.33$.

\subsection{A general lemma}\label{sec:Bprime}
With $w=\sqrt{1-x^2}$ and $g=x$, part (a) gives the positivity in Lemma~\ref{lem:turan} without Bessel functions.

\begin{lemma}[tilted covariances for even weights; proven]\label{lem:Bprime}
Let $w\ge0$ be an even integrable function on $[-1,1]$, not a.e.\ zero on $(0,1]$, let $g\in C^1[-1,1]$ satisfy $g'(x)+g'(-x)>0$ for $x\in[0,1]$, and for $\lambda\in\RR$ let $\mu_\lambda(dx)=Z^{-1}w(x)e^{\lambda g(x)}dx$.
\begin{enumerate}[label=\textup{(\alph*)},leftmargin=2em]
\item $\Cov_{\mu_\lambda}(x,x^2)>0$ for $\lambda>0$ and $\Cov_{\mu_\lambda}(x,x^2)<0$ for $\lambda<0$.
\item If moreover $g'(x)>0$ on $[-1,1]$ and $g'(x)\ge g'(-x)$ for $x\in[0,1]$, then $\Cov_{\mu_\lambda}(g,x^2)>0$ for $\lambda>0$; in particular, $\lambda\mapsto\E_{\mu_\lambda}[x^2]$ is strictly increasing on $(0,\infty)$.
\end{enumerate}
\end{lemma}

\begin{proof}
We first take $\lambda>0$. \emph{Folding.} For $x\in[0,1]$ put $A(x)=e^{\lambda g(x)}$, $B(x)=e^{\lambda g(-x)}$, let $\tilde m$ be the finite measure $w(x)(A(x)+B(x))\,dx$ on $[0,1]$ and $m=\tilde m/\tilde m([0,1])$; $m$ is absolutely continuous and $m((0,1])>0$. For bounded measurable $F$ on $[-1,1]$ define $\bar F(x)=(F(x)A(x)+F(-x)B(x))/(A(x)+B(x))$ on $[0,1]$. Since $w$ is even, $\int_{-1}^1F\,w\,e^{\lambda g}\,dx=\int_0^1\bar F\,d\tilde m$, so $\E_{\mu_\lambda}[F]=\E_m[\bar F]$; and since $x^2$ is even, $\overline{Fx^2}=x^2\bar F$. Hence $\Cov_{\mu_\lambda}(F,x^2)=\Cov_m(\bar F,x^2)$.

\emph{Chebyshev.} If $\phi$ is non-decreasing on $[0,1]$ and not $m$-a.e.\ constant, then $\Cov_m(\phi,x^2)=\frac12\iint(\phi(x)-\phi(y))(x^2-y^2)\,dm(x)dm(y)>0$: the integrand is $\ge0$, and it vanishes $m\otimes m$-a.e.\ only if $\phi(x)=\phi(y)$ for $m\otimes m$-a.e.\ $(x,y)$ (the diagonal is null because $m$ has no atoms), i.e.\ only if $\phi$ is $m$-a.e.\ constant. So it suffices to show that $\bar x$ and $\bar g$ are strictly increasing on $(0,1]$; a function strictly increasing on $(0,1]$ is not a.e.\ constant for an atomless measure with $m((0,1])>0$.

\emph{(a), $\lambda>0$: $F=x$.} Put $\delta(x)=g(x)-g(-x)$. Then $\delta(0)=0$ and $\delta'(x)=g'(x)+g'(-x)>0$, so $\delta>0$ on $(0,1]$ and $\delta$ is increasing. Now $\bar x=x\,(A-B)/(A+B)=x\tanh(\lambda\delta(x)/2)$ is the product of the non-negative strictly increasing function $x$ and the positive non-decreasing function $\tanh(\lambda\delta/2)$, hence strictly increasing on $(0,1]$: for $0<x<y$, $x\tanh(\lambda\delta(x)/2)<y\tanh(\lambda\delta(x)/2)\le y\tanh(\lambda\delta(y)/2)$.

\emph{(b): $F=g$.} Write $a=g'(x)$, $b=g'(-x)$, so $a\ge b>0$ on $[0,1]$, $A'=\lambda aA$, $B'=-\lambda bB$, and $A\ge B$ on $[0,1]$ (as $\delta\ge0$). With $\bar g=(g(x)A+g(-x)B)/(A+B)$, a direct differentiation gives
\[
(A+B)^2\,\bar g'(x)=a\,A\bigl[A+B(1+\lambda\delta)\bigr]-b\,B\bigl[B+A(1-\lambda\delta)\bigr]
\ \ge\ b\bigl(A^2-B^2+2\lambda\delta AB\bigr)\ \ge0,
\]
using $a\ge b$ and $A[A+B(1+\lambda\delta)]\ge0$ for the first inequality; the last expression is $>0$ for $x\in(0,1]$ because $\delta>0$ there. So $\bar g$ is strictly increasing on $(0,1]$.

\emph{(a), $\lambda<0$.} Let $Y=-x$ under $\mu_\lambda$. Then $Y$ has law $\propto w(y)e^{|\lambda|\tilde g(y)}dy$ with $\tilde g(y)=-g(-y)$, and $\tilde g'(y)+\tilde g'(-y)=g'(-y)+g'(y)>0$; by the case already proved, $\Cov(Y,Y^2)>0$, so $\Cov_{\mu_\lambda}(x,x^2)=-\Cov(Y,Y^2)<0$.
\end{proof}

\subsection{The Euler-factor tilt at every prime}\label{sec:euler-tilt}
For a prime $p$ put $t=1/\sqrt p$, $h_p(x)=1-2tx+t^2=|1-e^{i\theta}/\sqrt p|^2$ (so $L_p(f,1;\theta)=h_p(\cos\theta)^{-1}$ is the local Euler factor at the centre), $g_p=-\log h_p$, $\mu^{(p)}_\lambda(dx)\propto\sqrt{1-x^2}\,h_p(x)^{-\lambda}dx=\sqrt{1-x^2}\,e^{\lambda g_p(x)}dx$ and $\alpha_p(\lambda)=\Cov_{\mu^{(p)}_\lambda}(g_p,x^2)=\frac{d}{d\lambda}\E_{\mu^{(p)}_\lambda}[x^2]$.

\begin{corollary}[Euler-factor tilt; proven]\label{cor:euler-tilt}
For every prime $p$ \textup{(}indeed every $t\in(0,1)$\textup{)}: for $\lambda>0$, $\alpha_p(\lambda)>0$ and $\Cov_{\mu^{(p)}_\lambda}(x,x^2)>0$; for $\lambda<0$, $\Cov_{\mu^{(p)}_\lambda}(x,x^2)<0$. In particular, $\Cov_{\mu^{(p)}_\lambda}(\cos\theta,\cos2\theta)$ has the sign of $\lambda$.
\end{corollary}
\begin{proof}
$g_p'(x)=2t/h_p(x)>0$ on $[-1,1]$ (as $h_p\ge(1-t)^2>0$), and $g_p'$ is increasing in $x$, so $g_p'(x)\ge g_p'(-x)$ for $x\ge0$; apply Lemma~\ref{lem:Bprime}(a) and (b) with $w=\sqrt{1-x^2}$. The rest is $\cos2\theta=2x^2-1$.
\end{proof}

For $\lambda<0$ the corollary does not give the sign of $\alpha_p$, which is positive for small $|\lambda|$: $\alpha_p(0)=\Cov_{\mu_0^{(p)}}(g_p,x^2)>0$, as the odd part of $g_p$ is uncorrelated with $x^2$ under the semicircle and its even part $-\frac12\log\bigl((1+t^2)^2-4t^2x^2\bigr)$ is increasing in $x^2$.

v1 stated the Euler-factor tilt $\alpha_p(\lambda)>0$ ($\lambda>0$) as two theorems: for $p\ge p_0(\lambda)$ (v1, Theorem~10), and for fixed $p$ (v1, Theorem~11), whose proof stops at the inequality $\Cov_\lambda((\log h_p)^2,x^2)>0$; v1 verified $\alpha_p>0$ numerically for $p\le97$ and left that inequality open. The proof of Theorem~10 is not valid as written: its leading term $2t\,\Cov_{\mu_\lambda}(x,x^2)$, with $\mu_\lambda$ the linear tilt at the same $\lambda$, is of order $t$, whereas $\alpha_p(\lambda)=t^2(\lambda/4+1/8)+O(t^3)$ as $t\to0$ (at $p=97$, $\lambda=1/2$: $0.002577$ by quadrature, $0.002577$ from the expansion, $0.00615$ from v1's term; ancillary file \texttt{sec3/v1\_thm10\_expansion.py}, which also derives the expansion). Corollary~\ref{cor:euler-tilt} proves $\alpha_p(\lambda)>0$ for every $\lambda>0$ at every prime, without either argument.

\section{Theorem C}\label{sec:thmC}

v1's Theorem~C asserts that for weight-$2$ newforms $f$ of squarefree level $N$ with $L(f,1)>0$, fixed $p$, $c$, $\eta>0$, and $N\to\infty$ through squarefree values with $p\nmid N$, $\Cov(a_p(f),\Omega_{E_f}\mid L(f,1)\in[c,c+\eta],N)=C(c)/\sqrt p+o(1)$, where $C(c)$ is the prediction of an independent Sato--Tate model and changes sign at $c^*\approx1.5$, and that ``every ingredient (Shimura, Petersson, Weil, Deligne) is unconditional''. It is withdrawn; the gaps, in decreasing order of severity, are not repairable within that argument.

\begin{enumerate}[label=\textup{(\alph*)},leftmargin=2em,itemsep=1pt]
\item \emph{$\Omega_{E_f}$ does not exist for most of the family, and the period relation used is wrong in direction.} Only the newforms with rational Hecke eigenvalues have an elliptic curve $E_f$. v1 used ``Shimura's relation'' in the form $\omega_f\propto\Omega_{E_f}/N$ for the harmonic weight $\omega_f\propto\|f\|^{-2}$, i.e.\ $\Omega\propto1/L(\Ad f,1)$. For the optimal curve $E$, with modular parametrisation $\phi\colon X_0(N)\to E$ and Manin constant $c_M$ defined by $\phi^*\omega_E=2\pi i\,c_M\,f(z)\,dz$ for the N\'eron differential $\omega_E$ (Watkins' normalisation \cite{Watkins2002}), with $\|f\|^2=\int_{\Gamma_0(N)\backslash\mathfrak H}|f(z)|^2\,dx\,dy$ (the norm in $\omega_f$), $4\pi^2c_M^2\|f\|^2=\deg(\phi)\,\covol(\Lambda_E)$, where $\Lambda_E$ is the period lattice of $\omega_E$ (Zagier \cite[\S1]{Zagier1985}). With $\Omega_E=\int_{E(\RR)}|\omega_E|$ and $\Omega_E^-$ the least positive imaginary period ($\Lambda_E\cap i\RR=\ZZ\,i\Omega_E^-$), $\covol(\Lambda_E)=\Omega_E\Omega_E^-/2$ whether $E(\RR)$ has one component or two. So $\Omega_E=8\pi^2c_M^2\|f\|^2/(\deg(\phi)\,\Omega_E^-)$ has $\|f\|^2$ in the numerator, not in the denominator, and involves the modular degree, which v1 does not mention. v1 also attributes $c_M=1$ for semistable curves to Mazur. For the optimal curve, Mazur \cite[Cor.~4.1]{Mazur1978} shows that a prime dividing $c_M$ has its square dividing $4N$; $c_M=1$ for semistable optimal curves is due to \v{C}esnavi\v{c}ius \cite{Cesnavicius}, after partial results of Abbes--Ullmo \cite{AbbesUllmo}. v1's Remark~6 also states that $c_M$ ``is known to equal~1 for optimal curves in each isogeny class''. That is Manin's conjecture, which is open in general; for every $E$, \v{C}esnavi\v{c}ius--Neururer--Saha prove that $c_M\mid6\deg(\phi)$, and $c_M\mid\deg(\phi)$ if $N$ is cube-free \cite[Thm.~1.2]{CNS}.
\item \emph{The linearisation is not a proof.} v1 writes $h(L(f,1))=h(\bar L)+h'(\bar L)(L(f,1)-\bar L)+R_f$ for a bump $h$ supported on $[c,c+\eta]$ and a fixed $\bar L$, and keeps the linear term. This controls $h(L(f,1))$ only where $|L(f,1)-\bar L|$ is small compared with $\eta$; elsewhere $R_f$ is as large as the terms kept, and v1's bulk bound $|\sum_{\rm bulk}\omega_fa_pM(f)R_f|\ll(\log N)^{2A}|\sum_f\omega_fa_pM(f)|+(\log N)^{2A}\cdot(\text{off-diagonal})$ (v1, (23)) takes $R_f$ out of a signed sum at its supremum, which is not valid; the remaining correlation of $a_pM(f)$ with $R_f$ is bounded only through its Petersson off-diagonal terms (v1, \S4.6.2).
\item \emph{The conditioning is not reached.} Conditioning on $L(f,1)\in[c,c+\eta]$ involves the joint law of $a_p(f)$ and $L(f,1)$, for instance through the mixed moments $\sum_f\omega_fa_p(f)^jL(f,1)^k$ for all $j,k$, which are not available unconditionally.
\item \emph{Further steps.} v1's family $\Fam_N$ ($L(f,1)>0$) consists of forms of root number $+1$, for which the approximate functional equation $L(f,1)=2\sum_ma_m(f)V(m/\sqrt N)/m$ holds; but v1 applies the Petersson formula, which concerns sums over a full orthogonal basis of weight-$2$ cusp forms of level $N$, to sums over $\Fam_N$ (v1, (3)), and restricting to root number $+1$ adds twisted terms to every Petersson sum (\cite[\S2]{ILS}) that v1 omits. The mollifier estimate $M(f)L(\Ad f,1)=1+O(D^{-1/2+\varepsilon})$ ``on average (unconditionally, from Phragm\'en--Lindel\"of applied to $L(\mathrm{sym}^2f,s)$ on the 1-line)'' is asserted, not derived. The reason given for the existence of the limit defining $C(c)$, that the Euler factors are $1+O(1/\ell)$, is false: $L_\ell(f,1;\theta)=1+2\cos\theta/\sqrt\ell+O(1/\ell)$.
\item \emph{Misattribution.} v1 credits Rouse--Thorner \cite{RT} with effective rates for the level-aspect equidistribution of Satake parameters; their paper concerns a fixed form as $p\to\infty$, and its main theorem is conditional: it assumes that the symmetric power $L$-functions of the form have the analytic properties predicted by Langlands functoriality and satisfy the generalised Riemann hypothesis \cite[Conj.~1.1, Thm.~1.2]{RT}.
\end{enumerate}

\section{The data, recomputed}\label{sec:empirical}

This section is \emph{computed} and replaces v1's \S3 and \S5.1; its scripts and outputs are ancillary files. BSD invariants are Cremona's \texttt{allbsd} tables \cite{CremonaEcdata}, used verbatim; at rank $0$ the tabulated $|\Sha|$ satisfies $|\Sha|=L(E,1)|T|^2/(\Omega\prod c_p)$ (maximal deviation of the logarithmic identity $10^{-14}$ for $N<10^5$). The traces $a_p$ are used for all $657{,}396$ curves with $N<10^5$; at each $p$ the curves with $p\mid N$ are excluded. The subsets are R0, rank $0$ and $N<10^5$ ($n=267{,}565$; $168{,}760$ isogeny classes); S1, R0 with $|\Sha|=1$ ($228{,}193$); S11, R0 with $|\Sha|=|T|=1$ ($109{,}944$). v1's ``$657{,}473$ rank-$0$ elliptic curves'' (\S3.1) and ``$657{,}000$ curves'' (abstract) are close to $657{,}396$, the number of curves of \emph{all} ranks with $N<10^5$; the rank-$0$ count is $267{,}565$.

\emph{Conditioning on $L(E,1)$.} ``Fine bins'': $100$ quantile bins of $L(E,1)$ ($\max(10,\lfloor n/200\rfloor)$ if $n<20{,}000$), both variables demeaned within bin, pooled Pearson correlation. ``v1 method'': $a_p$ regressed linearly on $L(E,1)$, residual correlated with $\Omega$ (v1, \S3.1). Standard errors are curve bootstraps with the bins fixed unless stated otherwise. As $a_p$ and $L(E,1)$ are isogeny invariants and the curves of an isogeny class are not independent, curve bootstraps understate the uncertainty; for the $\chi^2$ of \S\ref{sec:corrdecay} we also give values from a bootstrap over isogeny classes, and the conditioned decile correlations of \S\ref{sec:signflip} carry standard errors from such a bootstrap.

\begin{table}[htbp]
\centering
\caption{$\Corr(a_p,\Omega\mid L(E,1)\text{ bins})$, rank $0$, $N<10^5$, $p\nmid N$; curve-bootstrap standard errors \textup{(}$200$ replicates\textup{)}, which understate the uncertainty and so inflate $\chi^2$ \textup{(}isogeny-class values in \S\ref{sec:corrdecay}\textup{)}. Left: by prime, and $\chi^2$ of the $C/\sqrt p$ fit over $p\le97$. Right: mean over $p\le13$, by conductor.}\label{tab:corr}
\footnotesize
\setlength{\tabcolsep}{3pt}
\begin{tabular}[t]{@{}rrrr@{}}
\toprule
$p$ & R0 & S1 & S11 \\
\midrule
2 & $+0.039\pm0.004$ & $+0.043\pm0.005$ & $+0.014\pm0.006$ \\
3 & $+0.062\pm0.004$ & $+0.076\pm0.004$ & $+0.058\pm0.005$ \\
5 & $+0.047\pm0.003$ & $+0.062\pm0.003$ & $+0.036\pm0.004$ \\
7 & $+0.032\pm0.003$ & $+0.046\pm0.003$ & $+0.033\pm0.003$ \\
11 & $+0.012\pm0.002$ & $+0.023\pm0.002$ & $+0.011\pm0.004$ \\
13 & $+0.021\pm0.002$ & $+0.031\pm0.002$ & $+0.020\pm0.003$ \\
31 & $+0.011\pm0.002$ & $+0.018\pm0.002$ & $+0.012\pm0.003$ \\
\midrule
$\chi^2$ & $280/24$ & $325/24$ & $234/24$ \\
\bottomrule
\end{tabular}
\hfill
\begin{tabular}[t]{@{}lrrr@{}}
\toprule
$N$ & R0 & S1 & S11 \\
\midrule
${<}5\mathrm{K}$ & $+0.064\pm0.005$ & $+0.065\pm0.005$ & $+0.022\pm0.007$ \\
$5$--$15\mathrm{K}$ & $+0.032\pm0.004$ & $+0.039\pm0.004$ & $+0.015\pm0.005$ \\
$15$--$35\mathrm{K}$ & $+0.023\pm0.003$ & $+0.034\pm0.003$ & $+0.019\pm0.004$ \\
$35$--$65\mathrm{K}$ & $+0.023\pm0.002$ & $+0.035\pm0.002$ & $+0.024\pm0.003$ \\
$65$--$100\mathrm{K}$ & $+0.021\pm0.002$ & $+0.036\pm0.002$ & $+0.023\pm0.003$ \\
\bottomrule
\end{tabular}
\end{table}

\subsection{The period--trace correlation}\label{sec:corrdecay}
On R0 the fine-bin correlation at $p=3$ is $+0.062\pm0.004$ (Table~\ref{tab:corr}). v1's ``$+0.097$ at $p=3$ (all rank-$0$ curves)'' matches most closely the S11 subset under the v1 method ($+0.0965$); the v1 method gives $+0.093$ on S1 and $+0.078$ on R0. The correlation is positive at every $p\le31$ on every subset; over the $25$ primes $p\le97$ it is $\le0$ at $3$ primes on R0, $2$ on S1 and $7$ on S11. For v1's $C=0.166$ in ``decays as $C/\sqrt p$'' (v1, \S3.2 and Figure~1), fitting $C/\sqrt p$ to the v1-method correlations at the $25$ primes $p\le97$ gives $C=0.163$ on S1 with weights $n$ (${\approx}1/\mathrm{SE}^2$) and $0.145$ by unweighted least squares ($0.127$ and $0.117$ on R0, $0.156$ and $0.141$ on S11); neither fit gives $0.166$ on any subset. The $C/\sqrt p$ \emph{shape} describes the prime dependence poorly: a weighted least-squares fit (weights $1/\mathrm{SE}^2$, no intercept) of the fine-bin correlations over $p\le97$ gives $\chi^2\approx280$, $325$, $234$ on $24$ degrees of freedom for R0, S1, S11. With standard errors from a bootstrap over isogeny classes ($300$ replicates; on average $1.23$, $1.21$, $1.03$ times the curve-bootstrap ones) it gives $196$, $235$, $214$, and a generalised least-squares fit with the full $25\times25$ covariance of that bootstrap gives $214$, $254$, $227$ ($195$, $232$, $208$ after the factor $(R-27)/(R-1)$ that corrects the inverse of a covariance estimated from $R=300$ replicates; ancillary file \texttt{16\_chi2\_cluster.py}).

v1 reported the correlation ``flat across five conductor windows'' ($n=19$K, $37$K, $55$K, $57$K, $60$K; $+0.059$, $+0.047$, $+0.048$, $+0.049$, $+0.043$; v1, \S3.3). On R0 it is not flat: $+0.064\pm0.005$ for $N<5{,}000$, falling to $+0.021\pm0.002$ for $N\ge65{,}000$ (Table~\ref{tab:corr}, right). On S1 it is flat only from $N\ge5{,}000$ on. v1's row counts match none of the subsets in \texttt{g2\_window\_counts\_by\_subset.csv}, but they add up to $228$K, the size of S1 and of v1's sample in its \S3.2, and v1 used its own method; the comparable series is therefore S1 under the v1 method: $+0.078\pm0.005$, $+0.058\pm0.004$, $+0.052\pm0.003$, $+0.051\pm0.002$, $+0.051\pm0.002$ (curve bootstrap; ancillary file \texttt{15\_windows\_v1method.py}). It too is not flat below $N=5{,}000$.

\subsection{The sign flip}\label{sec:signflip}
v1's Table~1 gives $\Corr(a_3,1/\prod c_p)$ by decile of $L(E,1)$ on the curves with $|\Sha|=|T|=1$: positive in deciles $1$--$9$ and $-0.050$ in decile $10$, ``confirming'' a predicted sign change. With equal-count deciles ($3\nmid N$, $n=40{,}489$) the values are positive in deciles $1$--$9$ and $-0.065\pm0.017$ in decile $10$, where $\Corr(a_3,\Omega)$ is $+0.006$. But decile $10$ is wide: inside it $L(E,1)$ runs over $[5.64,18.3]$, and both variables are correlated with $\log L(E,1)$ ($+0.31$ for $a_3$, $-0.31$ for $1/\prod c_p$). Conditioning on $L(E,1)$ inside the decile (both variables demeaned within $10$ equal-count sub-bins) gives $+0.034\pm0.018$, with a bootstrap over isogeny classes. The flip at $p=3$ is an artefact of conditioning on deciles.

\begin{table}[htbp]
\centering
\caption{Conditioned correlations in deciles $1$ and $10$ of $L(E,1)$ \textup{(}ten sub-bins inside each decile\textup{)}, with $z$ from a bootstrap over isogeny classes \textup{(}$300$ replicates\textup{)}; $p\nmid N$. From \texttt{g3c\_decile\_primes\_cluster.csv}.}\label{tab:deciles}
\footnotesize
\setlength{\tabcolsep}{3pt}
\begin{tabular}{@{}rrcccccc@{}}
\toprule
 & & \multicolumn{2}{c}{R0} & \multicolumn{2}{c}{S1} & \multicolumn{2}{c}{S11} \\
\cmidrule(lr){3-4}\cmidrule(lr){5-6}\cmidrule(l){7-8}
$p$ & decile & $1/\prod c_p$ & $\Omega$ & $1/\prod c_p$ & $\Omega$ & $1/\prod c_p$ & $\Omega$ \\
\midrule
7 & 1 & $+0.028$\,$(+2.9)$ & $+0.049$\,$(+4.6)$ & $+0.029$\,$(+2.8)$ & $+0.058$\,$(+5.4)$ & $+0.037$\,$(+3.1)$ & $+0.046$\,$(+4.2)$ \\
7 & 10 & $-0.046$\,$(-5.4)$ & $-0.042$\,$(-4.6)$ & $-0.029$\,$(-3.0)$ & $-0.029$\,$(-3.0)$ & $-0.037$\,$(-3.1)$ & $-0.036$\,$(-3.1)$ \\
11 & 1 & $+0.043$\,$(+5.0)$ & $+0.048$\,$(+4.2)$ & $+0.041$\,$(+4.6)$ & $+0.048$\,$(+4.1)$ & $+0.034$\,$(+2.8)$ & $+0.036$\,$(+3.1)$ \\
11 & 10 & $-0.028$\,$(-3.6)$ & $-0.053$\,$(-7.1)$ & $-0.056$\,$(-6.2)$ & $-0.052$\,$(-5.9)$ & $-0.046$\,$(-4.2)$ & $-0.047$\,$(-4.4)$ \\
\bottomrule
\end{tabular}
\end{table}

\emph{Other primes.} With the same conditioning and isogeny-class standard errors, the ancillary file \texttt{g3c\_decile\_primes\_cluster.csv} gives deciles $1$ and $10$ at $p=2,3,5,7,11,13$ (summarised by \texttt{14\_decile\_pattern\_summary.py}); at each prime the deciles are recomputed after excluding $p\mid N$. At $p=7$ and $p=11$ the shape v1 predicted, positive at small $L(E,1)$ and negative at large $L(E,1)$, holds on every subset and for both $1/\prod c_p$ and $\Omega$ (Table~\ref{tab:deciles}): all $24$ cells have the predicted sign, decile $1$ from $+0.028$ to $+0.058$ and decile $10$ from $-0.028$ to $-0.056$, with $|z|$ from $2.8$ to $7.1$. The file holds $120$ tests, of which $72$ concern deciles $1$ and $10$ ($6$ primes $\times$ $3$ subsets $\times$ $2$ deciles $\times$ $2$ variables). Of the $24$ cells, $15$ exceed the two-sided Bonferroni threshold at level $5\%$ for these $72$ tests ($|z|>3.39$), and the same $15$ exceed it for all $120$ ($|z|>3.53$); they include $7$ of the $8$ cells on R0, the exception being $p=7$, decile $1$, $1/\prod c_p$ ($z=2.9$). The cells are not independent: R0 $\supset$ S1 $\supset$ S11, and on S11, where $\Omega=L(E,1)/\prod c_p$, the two variables nearly coincide within a sub-bin. At the other primes the pattern is weak, partial or reversed. At $p=5$ all $12$ cells have the predicted sign, but $|z|\le2.33$. At $p=13$ decile $1$ is positive ($z$ from $2.3$ to $5.4$) but decile $10$ is null ($|z|\le0.42$). At $p=3$ there is no flip: for $1/\prod c_p$ on R0 and S1 the signs are reversed, beyond the Bonferroni threshold only in decile $10$ (decile $1$: $-0.030$, $-0.039$, $z=-2.2$, $-2.6$; decile $10$: $+0.107$, $+0.121$, $z=8.9$, $8.2$), for $\Omega$ both deciles are positive on R0 and S1, and on S11 no decile-$1$ or decile-$10$ cell reaches $|z|=2$. At $p=2$ the signs are reversed on R0 and S1 (beyond the Bonferroni threshold only for $1/\prod c_p$ in decile $10$ on S1: $+0.075$, $z=4.1$), and the predicted signs on S11 have $|z|\le2.91$. So the shape that v1 predicted is in the data at $p=7$ and $p=11$, on every subset, though not at $p=3$, where v1 claimed it, and not uniformly in $p$.

\subsection{The channel decomposition}\label{sec:channels}
At rank $0$, $\log\Omega=\log L(E,1)+2\log|T|-\log\prod c_p-\log|\Sha|$ is an identity (it defines the tabulated $|\Sha|$). So within an $L$-bin, $\Cov(a_p,\log\Omega\mid L)$ is the sum of the covariances of $a_p$ with $2\log|T|$, $-\log\prod c_p$, $-\log|\Sha|$ and a within-bin $\log L$ term, and a period--trace covariance at fixed $L(E,1)$ is a statement about $(|T|,\prod c_p,|\Sha|)$. On R0 at $p=3$ the Tamagawa term carries a fraction $1.54$ $[1.43,1.67]$ of the covariance ($95\%$ curve-bootstrap interval), the torsion term $-0.39$ and the $\Sha$ term $-0.15$; at $p=7$, $1.71$, $-0.16$, $-0.55$. On S11 the Tamagawa fraction is $1.00$, because there the identity reads $\Omega=L(E,1)/\prod c_p$. This is v1's ``at fine $L(1)$-conditioning, the Tamagawa channel accounts for $100\%$ of the signal'' (v1, abstract; cf.\ \S5.1): an identity on the subset where it was computed. v1's \S5.1 says as much: its Tamagawa-to-$\Omega$ ratio is ``converging to unity as guaranteed by $\Omega=L(1)/\prod c_v$ at $|\Sha|=|T|=1$''. The fault is that v1's abstract presents the identity as an empirical finding (``Empirically, the covariance is concentrated entirely in the Tamagawa product''). v1's ``Tamagawa-to-$\Omega$ signal ratio'' on S11, ``$1.003$ at $50$ bins and $1.004$ at $100$ bins, converging to unity'', recomputed as $\Corr(a_3,1/\prod c_p)/\Corr(a_3,\Omega)$, is $0.978$, $1.047$, $1.051$, $1.053$ at $10$, $50$, $100$, $200$ bins.

\subsection{The cross-validation claim}\label{sec:crossval}
v1 (\S5.1, abstract) states that its ``second-order adjoint response'' $R_2^{(\Omega)}(p)=-2p[(p+1)^2+a_p^2]/((p+1)^2-a_p^2)^2$ ``carries independent Tamagawa information beyond a full trace basis'', with cross-validated gains for $\log\prod c_p$ of $4.9\%$ over ``a full $25$-prime linear $a_p$ basis'' and $1.2\%$ over a linear-plus-quadratic one. On R0, with five folds grouped by conductor and $R_2=\sum_{p\le97}R_2^{(\Omega)}(p)$, all features using $a_p$ at $p\mid N$ as stored ($0$ or $\pm1$) (v1 states neither how the primes are combined nor which are used), the gain is $5.34\%$ over the linear basis and $0.75\%$ over the linear-plus-quadratic one. Over the full additive trace basis, the one-hot encoding of $a_p$ at the $25$ primes, which spans the additive functions $\sum_pf_p(a_p)$ ($576$ columns; cross-validated $R^2=0.370$, against $0.053$ for the linear basis), it is exactly $0$, as it must be: $R_2$ is such an additive function, so it lies in the span of that basis. A placebo $\sum_pa_p^4/p^2$ gains $0.71\%$ over the linear basis. For $\log|\Sha|$ the gains are $0.01\%$ and $0.05\%$. Folds grouped by isogeny class give the same gains. \emph{Withdrawn as unverified:} v1's \S3.2 statistics (``$58$ of $100$ (conductor, prime) tests are significant'', ``$64/100$'', ``$24/100$'', adjusted correlations $+0.019$ and $+0.050$) and its ``five independent controls'' (\S3.5) rest on procedures that v1 does not specify and that could not be reconstructed; they are withdrawn, not asserted to be wrong.

\subsection{Other material of v1}\label{sec:rest}
The following parts of v1 are also withdrawn. Unverified and not used here: the degree-$3$ type counts and $|\Sha|$ table of its Appendix~B; its Remark~4 ($|\Sha|$ constant within types when $4+27D^2$ is not squarefree); the Monte Carlo table of $C(c)$ in its \S4.2 and the comparison of $+0.091$ with $+0.097$ (\S\S1.3, 4.2, 5.3); the shape correlation $r=+0.25$ (\S5.3). Refuted: the claim that the Kronecker factorisation holds ``for every family of elliptic curves over $\FF_q(t)$ with squarefree discriminant'' (after Theorem~1 and in \S2.5), by the degree-$7$ example of Section~\ref{sec:ff}.

Three empirical statements of v1 survive, in this paper's numbers. (i) After conditioning on $L(E,1)$, the period--trace correlation is positive at small $p$: at every $p\le31$ on every subset, $+0.062\pm0.004$ at $p=3$ on R0 (Table~\ref{tab:corr}). (ii) $R_2^{(\Omega)}$ improves the cross-validated fit of $\log\prod c_p$ over the linear $a_p$ basis ($5.34\%$) and over the linear-plus-quadratic one ($0.75\%$), though not over the full additive basis (\S\ref{sec:crossval}). (iii) The decile pattern, positive at small $L(E,1)$ and negative at large $L(E,1)$, at $p=7$ and $p=11$ (\S\ref{sec:signflip}).

\section*{Acknowledgments}
Claude (Anthropic) was used to assist with the computations, the mathematical arguments and the writing.

\end{document}